\documentclass[a4paper,11pt]{article}
\usepackage{url}                     
\usepackage{graphicx} 
\usepackage{mathrsfs}
\usepackage{latexsym}
\usepackage{amssymb, amstext, amsgen}
\usepackage[centertags,intlimits]{amsmath}
\usepackage{amsthm}
\usepackage[latin1]{inputenc}   
\newtheorem{Theorem}{Theorem}

\newtheorem{proposition}{Proposition}

\newtheorem{Lemma}{Lemma}
\newtheorem{corollary}{Corollary}

\def\Z{\ensuremath{\mathbb{Z}}}

\newcounter{mathitem}
\newenvironment{mathitem}
  {\begin{list}{$(\roman{mathitem})$}{
   \setcounter{mathitem}{0}
   \usecounter{mathitem}
   \setlength{\topsep}{0pt plus 2pt minus 0pt}
   \setlength{\parskip}{10pt plus 2pt minus 0pt}
   \setlength{\partopsep}{5pt plus 2pt minus 0pt}
   \setlength{\parsep}{0pt plus 2pt minus 0pt}
   \setlength{\leftmargin}{10pt}
   \setlength{\itemsep}{7pt plus 2pt minus 0pt}}}
 {\end{list}}
 
\newcommand\bmi{\begin{mathitem}}
\newcommand\emi{\end{mathitem}}

\begin{document}

\title{Radio Labelings of Distance Graphs}

\author{Roman \v Cada\footnote{Department of Mathematics, NTIS - New Technologies for the Information Society, and CE-ITI - European Center of Excellence, University of West Bohemia, Pilsen, Czech Republic, e-mail: $\lbrace$cadar, ekstein, holubpre$\rbrace$@kma.zcu.cz.} , Jan Ekstein$^{\ast}$ , P\v remysl Holub$^{\ast}$, Olivier Togni\footnote{Universit\' e be Bourgogne, LE2I UMR CNRS 6306, France, e-mail: otogni@u-bourgogne.fr.}}

\date{\today}

\maketitle

\begin{abstract}
Motivated by the Channel Assignment Problem, we study radio $k$-labelings of graphs.
A radio $k$-labeling of a connected graph $G$ is an assignment $c$ of non negative integers to the vertices of $G$ such that $$|c(x) - c(y)| \geq k+1 - d(x,y),$$
for any two vertices $x$ and $y$, $x\ne y$, where $d(x,y)$ is the distance between $x$
and $y$ in $G$. 

In this paper, we study radio $k$-labelings of distance graphs, i.e., graphs with the set $\Z$ of
  integers as vertex set and in which two distinct vertices $i, j \in \Z$ are adjacent if and only if $|i - j| \in D$. We give some lower and upper bounds for radio $k$-labelings of distance graphs with distance sets $D(1,2,\dots, t)$, $D(1,t)$ and $D(t-1,t)$ for any positive integer $t>1$.


\end{abstract}

\section{Introduction}

In wireless networks, an important task is the assignment of radio frequencies to transmitters in a way that avoids interferences of their signals. An interference of signals can occur when transmitters which are close apart receive close frequencies. This problem has been modelled mathematically in a variety of colorings and labelings of vertices of graphs, where vertices represent transmitters and edges indicate closeness of the transmitters. In this context, a general $L(p_1,p_2, \dots, p_t)$-labeling problem (see e.g. \cite{KKT2}) has been proposed: find a labeling of the vertices of a given graph $G$ such that labels of any two vertices of $G$ at distance $d$ differ by at least $p_d$. The aim is to minimize range (or span) of used frequencies, i.e., the difference between the smallest and the largest used label.

The problem of the $L(p_1,p_2, \dots, p_t)$-labeling appears to be difficult in general, hence many particular cases have been studied. Among all, labelings with constraints at two distances, particularly $L(2,1)$-labeling introduced by Griggs and Yeh in \cite{GY92}, have been the subject of many articles. In this paper, we focus on the radio $k$-labeling problem, which one can see as an extension of $L(2,1)$-labeling and also as a particular case of the general $L(p_1,p_2, \dots, p_t)$-labeling.

Let $G$ be a connected graph and let $k$ be an integer, $k\geq 1$.
The distance between two vertices $u$ and $v$ of $G$ is denoted by
$d_G(u,v)$ (or simply $d(u,v)$) and the diameter of $G$ is denoted by $\text{diam}(G)$. 
A {\em radio $k$-labeling} $c$ of $G$ is an assignment of non negative
integers to the vertices of $G$ such that $$|c(u)-c(v)|\geq
k+1-d(u,v),$$ for every two distinct vertices $u$ and $v$ of $G$.
The span of the function $c$ denoted by $\text{rl}{_k}(c)$, is
$\text{max}\{c(x)-c(y): x,y \in V(G) \}$. The \textit{radio
$k$-labeling number} $\text{rl}{_k}(G)$ of $G$ is the minimum span among
all radio $k$-labelings of $G$.

The study of radio $k$-labelings was initiated by Chartrand et al. \cite{kcol}. Quite few results are known concerning radio $k$-labelings. The radio $k$-labeling number for paths was first studied by Chartrand et al. \cite{kcol}, where lower and upper bounds were given. These bounds have been improved by Kchikech et al. \cite{KKT2}. 

Radio $k$-labelings have been investigated mainly for fixed values of $k$.
A radio $1$-labeling is a proper vertex-colouring and $\text{rl}{_1}(G)=\chi(G)-1$. 
For $k=2$, the radio $2$-labeling problem corresponds to the well studied $L(2,1)$-labeling problem as we mentioned above.
Large values of $k$ (close to the diameter of the graph) were also considered for radio $k$-labelings as radio labeling ($k=\text{diam}(G)$) and radio antipodal labeling ($k=\text{diam}(G)-1$).
The interested reader is referred to surveys ~\cite{CZ07,Pan09} and recent papers~\cite{MOTW11,SP12} for complementary results.

%

Note that the authors of \cite{ChCyc,ChRad,ChLab,kcol} assume that the labels (colours) are positive. However, when speaking about labelings in relation with frequency assignment, it is more common to use non negative integers as labels. Thus the notation of the present paper follows the terminology of~\cite{KKT2,KKT,LiTre,LiXie1,LiMul,MOTW11,SP12} in which vertices are labeled by non negative integers.

One of motivations for the class of distance graphs considered in this paper comes from networks. In \cite{Wong}, Wong and Coppersmith introduced a concept of multiloop networks $ML(N,s_1,\dots, s_l)$ for organizing multimodule memory devices. In the graph terminology, such a network can be viewed as a graph with $N$ nodes $0,1,\dots, N-1$ and $lN$ links of $l$ types, where the $i$-type links ($i=1,\dots, l$) are $i,j$-edges where $i-j \equiv s_i \, (\mbox{mod }N)$. Such graphs are usually called circulant graphs. For infinitely large $N$, these multiloop networks become graphs which are called distance graphs. In fact, circulant graphs coincide exactly with the regular distance graphs as it was shown in \cite{Rautenbach}. Distance graphs were introduced by Eggleton et al. in \cite{Eggl}, a lot of papers on different kinds of colorings of distance graphs have been published in last 20 years, including some results for $L(2,1)$-labelings (see \cite{LWG08,TG04,TGK05}).

Let $D = \{d_{1}, d_{2}, ..., d_{k}\}$, where $d_{i}$ ($i = 1, 2, ..., k$) are positive integers such that
$d_{1} < d_{2} < ... < d_{k}$. The (infinite) \emph{distance graph} $D(d_{1}, d_{2}, ..., d_{k})$ has the set $\mathbb{Z}$ of integers 
as a vertex set and in which two distinct vertices $i, j \in \mathbb{Z}$ are adjacent if and only if $|i - j| \in D$. 
The {\it finite distance graph} $D_n(d_{1}, d_{2}, ..., d_{k})$ is the subgraph of $D(d_{1}, d_{2}, ..., d_{k})$ induced by vertices $0,1,\ldots ,n-1$. 

In \cite{LWG08,TG04,TGK05}, radio $2$-labeling numbers have been determined only for some of the distance graphs (mainly $4$-regular). The aim of this paper is to obtain bounds on the radio $k$-labeling number of some distance graphs in terms of $k$ (and not depending on the order of the graph). For any positive integer $t\geq 2$ we show that 
$$
\begin{array}{ccc}
\frac t2 k^2 + \frac 12 \leq & \hspace{-2mm} \text {rl}_k(D(1,2,\dots, t))\leq & \hspace{-2mm} \left\{ \begin{array}{ll} \frac t2 k^2 +\frac t2 k, & \mbox{when } k \mbox{ is odd},\\
 \frac t2 k^2 + k, & \mbox{when } k \mbox{ is even}. 
\end{array}
 \right.
\end{array}
$$

In Propositions \ref{p1t} and \ref{pt-1t} we give analogous lower bounds for the radio $k$-labeling number of distance graphs $D(1,t)$ and $D(t-1,t)$ for $k\geq \frac t2$ and $k\geq t$, respectively. 

When $k$ is odd, the upper bounds for $\mbox{rl}_k \left( D(1,2,\dots, \, t)\right)$ can be decreased for the distance graphs $D(1,t)$ and $D(t-1,t)$ as subgraphs of $D(1,2,\dots, t)$, $t> 2$, as it is shown in Theorems \ref{thm D(1,t) t odd}, \ref{thm D(1,t) t even} and \ref{thm D(t-1,t)}.







\section{Lower bounds}

A classical method for finding a lower bound on the radio  $k$-labeling number of a graph is to use the following relation with another graph parameter called the {\em upper traceable number}~\cite{OZV08}, denoted $t^+$: for a graph $G$ of order $n$ and for a linear ordering $s:(x_1,x_2,\dots,x_n)$ of its vertices, let $d(s)= \sum_{i=1}^{n-1}d(x_i,x_{i+1})$. Then $t^+(G)=\max{d(s)}$, where the maximum is taken over all linear orderings $s$ of vertices of $G$.

\begin{Theorem}[\cite{KKT2}]
\label{th1} For any integer $k\geq 1$, and any graph $G$ of order $n$,
$$\text{rl}_{k}(G)\geq
    (n-1)(k+1)-t^+(G).$$
\end{Theorem}

In order to find bounds for the upper traceable number of some distance graphs, the upper traceable number of the path (determined in \cite{LiMul}, without using the above terminology) will be useful:
\begin{Lemma}[\cite{LiMul}]
\label{t+Pn} For any integer $n\geq 2$,
    $$t^+(P_n)=\left\{%
\begin{array}{ll}
  \frac{1}{2}n^2-1 & \hbox{if $n$ is even,}    \\
  \frac{1}{2}(n^2-1)-1 & \hbox{if $n$ is odd.}  \\
\end{array}%
\right.$$
\end{Lemma}

\begin{Lemma}
\label{le1}
Let $G$ be a graph of order $n$ with $V(G)=\{0,1,\ldots ,n-1\}$. If there are positive real numbers $\alpha$ and $\beta$ such that $d_G(i,j)\le \frac{j-i+\alpha}{\beta}$ for any $i$ and $j$ satisfying $0\leq i< j\leq n-1$, then 

$$t^+(G)\le \frac{\frac{n^2}{2} +\alpha (n -1) -1}{\beta}.$$
\end{Lemma}

\begin{proof} For a path $P$ on the vertices $0,1,\ldots, n-1$, we have  $d_{P}(i,j)=|j-i|$ for every $i,j\in\{0,1,\ldots,n-1\}$.
Hence, for any ordering $(x_1,x_2,\dots,x_n)$ of the vertices of $G$, we have
$$\sum_{i=1}^{n-1} d_G(x_{i},x_{i+1}) \le \sum_{i=1}^{n-1}\frac{|x_{i+1}-x_{i}|+\alpha}{\beta}= \sum_{i=1}^{n-1}\frac{d_{P}(x_{i},x_{i+1})+\alpha}{\beta}\le \frac{t^+ (P_n) + \alpha(n-1)}{\beta}.$$

Therefore, with Lemma~\ref{t+Pn}, we obtain the desired inequality.
\end{proof}

\begin{proposition}\label{p12t} For any positive integers $k\ge 1$ and $t\ge 2$,
 $$\text{rl}_k(D(1,2,\ldots, t)) \ge \frac{t}{2}k^2 +\frac{1}{2}.$$
\end{proposition}

\begin{proof}
Let $n>t$ and let $G=D_n(1,2,\ldots,t)$. It is easily seen that for $j\ge i$,
$$d_G(i,j)=\left\lceil\frac{j-i}{t}\right\rceil \le \frac{j-i+t-1}{t}.$$

Hence $G$ satisfies the conditions of Lemma~\ref{le1} with $\alpha=t-1$ and $\beta=t$ and we have $$t^{+}(G)\le \frac{\frac{n^2}{2} + (t-1)(n-1)-1}{t}=\frac{n}{t}(\frac{n}{2}+t-1)-1.$$

Consequently, by Theorem~\ref{th1}, we obtain that
$$\text{rl}_k(G) \ge (n-1)(k+1) - \frac{n}{t}(\frac{n}{2}+t-1)+1 = n(k-\frac{n-2}{2t}) -k.$$ 

The right hand side of the inequality is maximized when $n=tk+1$ and it gives
$$\text{rl}_k(D_{tk+1}(1,2,\ldots,t)) \ge (tk+1)(k-\frac{tk-1}{2t}) -k=\frac{t}{2}k^2 +\frac{1}{2t}.$$
As $D_{tk+1}(1,2,\ldots,t)$ is a subgraph of $D(1,2,\ldots,t)$ and since the radio  $k$-labeling number is a natural number, we have $\text{rl}_k(D(1,2,\ldots,t))\ge\lceil\frac{t}{2}k^2 +\frac{1}{2t}\rceil \ge \frac{t}{2}k^2 +\frac{1}{2}$ and the desired inequality is proved.
\end{proof}

For the graphs $D(1,t)$, we can show a lower bound of the same order by using a similar argument. We will use the following statement.

\begin{Lemma}[\cite{T11}]\label{ldist}
The distance between two vertices $i$ and $j$ of $D(1, t)$ is $d(i, j) = \min(q + r; q +
1 + t - r),$ where $\vert i - j\vert  = qt + r,$ with $0 \leq r < t$.
\end{Lemma}

\begin{proof} \cite{T11}
Without loss on generality we can assume that $j\geq i$. Any minimal path between $i$ and $j$ uses either $q$ $t$-edges (edges joining vertex $a$ with vertex $a+t$ for any $a\in \mathbb{Z}$)
and $r$ $1$-edges (edges joining vertex $a$ with vertex $a+1$ for any $a\in \mathbb{Z}$) or $q + 1$ $t$-edges and $t-r$ $1$-edges. 
\end{proof}

\begin{proposition}\label{p1t} For any positive integers $t\ge 3$ and $k\ge \frac t2$,
 $$\text{rl}_k(D(1,t)) \ge \frac{t}{2}k^2 -P(t)k+Q(t),$$ with $P(t)=\frac{t^2}{2}-t+\frac{1}{2}$ and $Q(t)=\frac{t^3}{8}-\frac{t^2}{2}+\frac{3t}{4}-\frac{1}{2}$.
\end{proposition}

\begin{proof}
Let $n$ be a positive integer and let $G=D_n(1,t)$. Then, by Lemma \ref{ldist}, $d_G(i,j)=q+\min(r,t+1-r)$, for $j\ge i$ with $j-i=qt+r$ and $0\le r< t$.
Thus $$d_G(i,j)\le \frac{j-i+\left\lceil\frac{t}{2}\right\rceil(t-1)}{t}\le \frac{j-i+\frac{t+1}{2}(t-1)}{t}.$$

Hence $G$ satisfies the conditions of Lemma~\ref{le1} with $\alpha=\frac{t^2-1}{2}$ and $\beta=t$ and we have
$$t^{+}(G)\le \frac{\frac{n^2}{2} -1+ \frac{t^2-1}{2}(n-1)}{t} = \frac{(n-1)(n+t^2)-1}{2t}.$$

Consequently, by Theorem~\ref{th1}, we obtain that
$$\text{rl}_k(G) \ge (n-1)(k+1) - \frac{(n-1)(n+t^2)-1}{2t}= (n-1)(k+1-\frac{n+t^2}{2t}) +\frac{1}{2t}.$$ 

The right hand side of the inequality is maximized when $n=tk-\lfloor\frac{t^2}{2}\rfloor+t$ and since, by the hypothesis, $k\ge \frac t2$, we get $n\ge 1$. Reporting this value in the above inequality gives
$$\text{rl}_k(D_{tk-\left\lfloor\frac{t^2}{2}\right\rfloor+t}(1,t)) \ge (tk-\left\lfloor\frac{t^2}{2}\right\rfloor+t-1) (k+1-\frac{tk-\left\lfloor\frac{t^2}{2}\right\rfloor+t+t^2}{2t})+\frac{1}{2t}.$$
After simplification, in both cases $t$ is odd and $t$ is even, we obtain 
$$\text{rl}_k(D_{tk-\lfloor\frac{t^2}{2}\rfloor+t}(1,t)) \ge \frac{t}{2}k^2 -(\frac{t^2}{2}-t+\frac{1}{2})k + \frac{t^3}{8}-\frac{t^2}{2}+\frac{3t}{4}-\frac{1}{2}+\frac{1}{2t},$$ which concludes the proof.
\end{proof}

For the graphs $D(t-1,t)$, we first compute an upper bound on the distance between two vertices:
\begin{Lemma}
\label{let-1t}
Let $t\geq 2$ be an integer, $i,j$ a pair of vertices of the graph $G=D(t-1, t)$ and  let $\vert i-j \vert =qt+r$,
where $q,r\in \mathbb{N}$, $0\leq r<t$. Then $d_G(i,j)\le q+t$.
\end{Lemma}

\begin{proof}Let $i,j$ be two integers with $j\ge i$.
If $0\le j-i\le t/2$ then   $d_G(i,j)\le 2(j-i) \le t$ since $j-i=(j-i)t - (j-i)(t-1)$. If $t/2< j-i\le t-1$ then
$d_G(i,j)\le 2(t-j+i)+1 \le 2t-t-1+1=t$ since $j-i= (t-j+i)(t-1)-(t-1-j+i)t$.
Now, if $j-i \ge t$ then $j-i=qt+r$, with $q,r\in \mathbb{N}$ and $0\leq r<t$. Hence, with the above, $d_G(i,j)\le
q+t$.
\end{proof}

\begin{proposition}\label{pt-1t} For any positive integers $t\ge 3$ and $k\ge t$,
 $$\text{rl}_k(D(t-1,t)) \ge \frac{t}{2}k^2 -P(t)k+Q(t),$$ with $P(t)=t^2-t+1$ and
$Q(t)=\frac{t^3}{2}-t^2+\frac 32 t -1$.
\end{proposition}

\begin{proof}
Let $n,t$ be integers, $t\ge 3$, and let $G=D_n(t-1,t)$. Then, for any integers $i,j$, $0\le i\le j\le n-1$, we have  $d_G(i,j)\leq\lfloor\frac{j-i}{t}\rfloor +t \le \frac{j-i+t^2}{t}$ by Lemma~\ref{let-1t}.

Hence $G$ satisfies the conditions of Lemma~\ref{le1} with $\alpha=t²$ and $\beta=t$ and we have 
$$t^{+}(G)\le \frac{\frac{n^2}{2} -1+ t^2(n-1)}{t}= (n-1)(t+\frac{n}{2t}) +\frac{n}{2t}-\frac 1t .$$

Consequently, by Theorem~\ref{th1}, we obtain that
$$\text{rl}_k(G) \ge (n-1)(k+1) - (n-1)(t+\frac{n}{2t}) -\frac{n}{2t}+\frac 1t = (n-1)(k+1-t-\frac{n}{2t})-\frac{n}{2t}+\frac 1t.$$

The right hand side of the inequality is maximized when $n=tk+t-t^2$ and since $k\ge t$, we get $n\ge 1$. Reporting this value in the above inequality gives
$$\text{rl}_k(D_{tk+t-t^2}(1,t)) \ge (tk+t-t^2-1)(k+1-t-\frac{tk+t-t^2}{2t})-\frac{tk+t-t^2}{2t}+\frac 1t.$$
After simplification, we obtain 
$$\text{rl}_k(D_{tk+t-t^2}(1,t)) \ge \frac{t}{2}k^2 -(t^2-t+1)k +
\frac{t^3}{2}-t^2+\frac 32 t-1+\frac 1t,$$ which concludes the proof.
\end{proof}

\section{Upper bounds}
\subsection{$D(1,2,\dots, t)$}

Recall that a labeling $c$ of vertices of a graph $G$ is a radio $k$-labeling if, for every pair $i,j$ of vertices of $G$,

\begin{equation} \label{eqn1}
\vert c_i-c_j\vert + d_G(i,j)>k,
\end{equation} 

where $c_i$, $c_j$ denote labels of $i$ and $j$ respectively.

\begin{Lemma}\label{lemma1-t}
Let $t\geq 2$ be an integer, $i,j$ a pair of vertices of the graph $G=D(1,2,\dots, t)$ and  let $\vert i-j \vert =qt+r$, where $q,r\in \mathbb{N}$, $0\leq r<t$. Then $d_G(i,j)=q$ if $r=0$ and $d_G(i,j)=q+1$ otherwise.
\end{Lemma}

\begin{proof}
Suppose that $j>i$. There is a path $P=i,i+t,i+2t, \dots, i+qt$ of length $q$ in $G$. If $r=0$ then $j=i+qt$ and hence $d_G(i,j)\leq q$, else there is an edge between vertices $i+qt,i+qt+r=j$ in $G$ and hence $d_G(i,j)\leq q+1$. Clearly there is no shorter $i,j$-path in $G$.
\end{proof}

\begin{Theorem} \label{thmkeven}
Let $t\geq 2$ be an integer, $k$ be an even positive integer and $G=D(1,2,\dots, t)$. Then
$$
\text{rl}_k(G)\leq \frac t2 k^2+k.  
$$

\end{Theorem}

\begin{proof}
First we define a periodical pattern of labels. For vertices $1,2,\dots, tk+3$ of the distance graph $G$ we set labels using the following table.

\begin{table}[ht]
\small
\centering
$$\begin{array}{|c|c|c|c|c|c|c|c|c|c|c|c|c|} 
\hline
\mbox{vertex} & 1 & 2 & 3 & 4 & \dots & \frac t2 k+ 2 & \frac t2 k+3 & \frac t2 k+4 & \frac t2 k+5 &\dots & tk+3 \\
\hline
\mbox{label} & 0 & k & 2k & 3k & \dots & \left( \frac t2 k+1 \right) k & \frac k2 & \frac k2 +k & \frac k2 + 2k & \dots & \frac k2 + \left( \frac t2 k\right) k \\
\hline
\end{array}$$
\normalsize
\caption{\label{patternkeven}A periodical pattern for $G=D(1,2,\dots, t)$ and even $k$.}
\end{table}

Then we can define a labeling $c$ of all vertices of $G$ setting $c(a+b(tk+3))=c(a)$, $ a\in \{1,2,\dots, tk+3\}$ and $b\in \mathbb{Z}$, i.e., we repeat the defined periodical pattern for all vertices of $G=D(1,2,\dots, t)$.

Now we show that the labeling $c$ is a radio $k$-labeling of $G$, i.e., the inequality (\ref{eqn1}) holds for every $i,j\in V(G)$. Note that the length of the pattern is $tk+3$. For every pair $i,j$ of vertices of $G$ with $\vert i-j\vert \geq tk+3$, it holds that $d_G(i,j)>k$. Therefore it suffices to prove that there is no conflict in labeling $c$ between vertices in two consecutive copies of the pattern. If $\vert c_i-c_j\vert>k$ then the inequality (\ref{eqn1}) trivially holds. Now we consider the following possibilities.

\bmi
\item[{\sl Case 1:}] {\sl $\vert c_i-c_j\vert=0$}. \\ By the definition of the pattern given by Table \ref{patternkeven} it follows that $\vert i-j\vert =tk+3$. Then $d_G(i,j)>k$, implying that the inequality (\ref{eqn1}) holds.
\item[{\sl Case 2:}] {\sl $\vert c_i-c_j\vert =k$}. \\ Then trivially $\vert i-j\vert >0$ and hence $d_G(i,j)>0$. This implies that the inequality (\ref{eqn1}) holds.
\item[{\sl Case 3:}] {\sl $0<\vert c_i-c_j\vert<k$}. \\ 
   From pattern given by Table \ref{patternkeven} we obtain $\vert c_i-c_j\vert = \frac k2$, and $\vert i-j\vert = \frac t2k+1$ or $\vert i-j\vert = \frac t2 k+2$. Then, by Lemma \ref{lemma1-t}, $d_G(i,j)>\frac k2$. Thus we have $\vert c_i-c_j\vert + d_G(i,j)>\frac k2+\frac k2=k$ and the inequality (\ref{eqn1}) holds.
\emi
We have shown that the defined labeling is a radio $k$-labeling of $G$. The maximum used label is $\frac t2 k^2+k.$
\end{proof}

\begin{Theorem}\label{thmkodd}
Let $t\geq 2$ be a positive integer, $k$ be an odd positive integer and let $G=D(1,2,\dots, t)$. Then
$$
\text{rl}_k(G)\leq \frac t2 k^2+\frac t2 k.  
$$
\end{Theorem}

\begin{proof}

First we define periodical patterns of labels. For even $t$, we set labels of vertices $1,2,\dots, tk+1$ of the distance graph $G$ by Table \ref{patternkodde}, where $l=k+1$.

\small
\begin{table}[ht]
\centering
$$\begin{array}{|c|c|c|c|c|c|c|c|c|c|c|c|c|} 
\hline
\mbox{vertex} & 1 & 2 & 3 & 4 & \dots & \frac t2 k+1 & \frac t2 k+2 & \frac t2 k+3 & \frac t2 k+4 &\dots & tk+1 \\
\hline
\mbox{label} & 0 & l & 2l & 3l & \dots & \left( \frac t2 k \right) l & \frac l2 & \frac l2 +l & \frac l2 + 2l & \dots & \frac l2 + \left( \frac t2 k - 1\right) l \\
\hline
\end{array}$$

\caption{\label{patternkodde}A periodical pattern for $G=D(1,2,\dots, t)$, odd $k$ and even $t$.}
\end{table}
\normalsize

Analogously, for odd $t$, we set labels of vertices $1,2,\dots, tk+ 1$ of the distance graph $G$ by Table \ref{patternkoddo}, where $l=k+1$.

\small
\begin{table}[ht]
\centering
$$\begin{array}{|c|c|c|c|c|c|c|c|c|c|c|c|c|} 
\hline
\mbox{vertex} & 1 & 2 & 3 & 4 & \dots & \frac t2 k+\frac 12 & \frac t2 k+\frac 32 & \frac t2 k+\frac 52 & \frac t2 k+\frac 72 &\dots & tk+1 \\
\hline
\mbox{label} & 0 & l & 2l & 3l & \dots & \left( \frac t2 k -\frac 12\right) l & \frac l2 & \frac l2 +l & \frac l2 + 2l & \dots & \frac l2 + \left( \frac t2 k - \frac 12 \right) l \\
\hline
\end{array}$$
\caption{\label{patternkoddo}A periodical pattern for $G=D(1,2,\dots, t)$, odd $k$ and odd $t>1$.}
\end{table}
\normalsize

Then, for any parity of $t$, we can define a labeling $c$ of all vertices of $G$ setting $c(a+b(tk+1))=c(a)$, $ a\in \{1,2,\dots, tk+1\}$ and $b\in \mathbb{Z}$, i.e., we repeat the defined periodical pattern for all vertices of $G=D(1,2,\dots, t)$.

We show that the labeling $c$ is a radio $k$-labeling of $G$, i.e., the inequality (\ref{eqn1}) holds for every $i,j\in V(G)$. Note that the length of both patterns is $tk+1$. For every pair $i,j$ of vertices of $G$ with $\vert i-j\vert \geq tk+1$, it holds that $d_G(i,j)>k$. Therefore it suffices to prove that there is no conflict in labeling $c$ between vertices in two consecutive copies of the pattern. If $\vert c_i-c_j\vert>k$ then the inequality (\ref{eqn1}) trivially holds. Now we consider the following possibilities.

\bmi
\item[{\sl Case 1:}] {\sl $\vert c_i-c_j\vert=0$}. \\ By the definition of the patterns given by Tables \ref{patternkodde} and \ref{patternkoddo} it follows that $\vert i-j\vert =tk+1$. Thus $d_G(i,j)>k$, implying that the inequality (\ref{eqn1}) holds.
\item[{\sl Case 2:}] {\sl $\vert c_i-c_j\vert =k$}. \\ Then trivially $\vert i-j\vert >0$ and hence $d_G(i,j)>0$. This implies that the inequality (\ref{eqn1}) holds.
\item[{\sl Case 3:}] {\sl $0<\vert c_i-c_j\vert<k$}. \\ From the patterns given by Tables \ref{patternkodde} and \ref{patternkoddo} we obtain $\vert c_i-c_j\vert = \frac l2$. Now we have two subcases depending on parity of $t$.
\bmi
   \item[{\sl \hspace*{3mm} Subcase 3.1:}] {\sl $t$ is even}. \\
   From pattern given by Table \ref{patternkodde} we get $\vert i-j\vert = \frac t2k=\frac{k-1}2t+\frac t2$ or $\vert i-j\vert = \frac t2 k+1=\frac{k-1}2 t+\frac t2+1$. Then, by Lemma \ref{lemma1-t}, $d_G(i,j)>\frac {k-1}2$. Since $l=k+1$, we have $\vert c_i-c_j\vert + d_G(i,j)>\frac {k+1}2+\frac {k-1}2=k$ and the inequality (\ref{eqn1}) holds.
   \item[{\sl \hspace*{3mm} Subcase 3.2:}] {\sl $t$ is odd}. \\
   From pattern given by Table \ref{patternkoddo} we have $\vert i-j\vert = \frac t2 k+\frac 12=\frac{k-1}2 t+ \frac t2+\frac 12$ or $\vert i-j\vert = \frac t2 k+\frac 32= \frac{k-1}2 t+ \frac t2+\frac 32$ or (for $t>1$) $\vert i-j \vert = \frac t2 k - \frac 12 = \frac{k-1}2 t+\frac t2-\frac 12$. In each of these  possibilities we have $d_G(i,j)>\frac{k-1}2$ by Lemma \ref{lemma1-t}. Since $l=k+1$, we have $\vert c_i-c_j\vert + d_G(i,j)>\frac {k+1}2+\frac {k-1}2=k$ and the inequality (\ref{eqn1}) holds.

\emi
\emi
We have shown that the defined labeling is a radio $k$-labeling of $G$. The maximum used label is $\frac t2 k^2+\frac t2 k.$
\end{proof}

Now we consider the particular case when $t=2$.
Lower and upper bounds for $\text{rl}_k(D(1,2))$ can be derived from radio labelings of the square of paths and cycles. Note that the square of a graph $G$ is the graph with $V(G^2)=V(G)$ in which two vertices are adjacent if their distance in $G$ is at most two, and is denoted by $G^2$.
Let $P_n^2$ and $C_n^2$ be the square of the path and of the cycle of order $n$, respectively.
The graph $P_n^2$ is a subgraph of $D(1,2)$, and one can use any radio $k$-labeling of $C_n^2$ as a pattern to label the distance graph $D(1,2)$.
Liu and Xie showed that $\text{rl}_k (P_{2k+1}^2)=k^2+2$ ~\cite{LiXie1} and $\text{rl}_k (C_{4k+1}^2)=k^2+2k$~\cite{LiXie2}. Thus we can obtain the following bounds for $\text{rl}_k (D(1,2))$

$$
k^2+2\leq \text{rl}_{k}(D(1,2))\leq k^2+2k.
$$

From Theorem \ref{thmkeven} and Theorem \ref{thmkodd} we obtain the following statement which strengthens the upper bound mentioned in the previous inequality.

\begin{corollary}
For any positive integer $k$, $$ k^2+2 \leq \text{rl}_{k}(D(1,2))\le  k^2+k.$$
\end{corollary}

\subsection{$D(1,t)$}
Now we focus on the distance graphs $G=D(1,t)$. For even $k$ we did not find any improvement of the upper bound for $\text{rl}_k(G)$ given by Theorem \ref{thmkeven}, but for odd $k$ we can improve the upper bound for $\text{rl}_k(G)$ from Theorem \ref{thmkodd}.

\begin{Theorem}\label{thm D(1,t) t odd} Let $t\geq 3$ and $k\geq 1$ be odd integers, let $G$ be a distance graph $D(1,t)$. Then
$\text{rl}_{k}(G)\leq \frac t2 k^{2}-\frac 12$. 
\end{Theorem}

\begin{proof}
First we define a periodical pattern of labels. For vertices $1,2,\dots, tk+1$ of the distance graph $G=D(1,t)$ we set labels using the following table.

\begin{table}[ht]
\centering
\small
$$\begin{array}{|c|c|c|c|c|c|c|c|c|c|c|c|c|} 
\hline
\mbox{vertex} & 1 & 2 & 3 & 4 &\dots& \frac t2 k+ \frac 12 & \frac t2 k+\frac 32 & \frac t2 k+ \frac 52 & \frac t2 k+\frac 72 &\dots& tk+1 \\
\hline
\mbox{label} & 0 & k & 2k & 3k &\dots& \left( \frac t2 k-\frac 12 \right) k & \frac{k-1}2 & \frac{k-1}2 +k & \frac{k-1}2 + 2k &\dots& \frac{k-1}2 + \left( \frac t2 k - \frac 12 \right) k \\
\hline
\end{array}$$
\normalsize

\caption{\label{toddkodd}A periodical pattern for $G=D(1,t)$ and odd $t\geq 3$.}
\end{table}

Then we can define a labeling $c$ of all vertices of $G$ setting $c(a+b(tk+1))=c(a)$, $ a\in \{1,2,\dots, tk+1\}$ and $b\in \mathbb{Z}$, i.e., we repeat the defined periodical pattern for all vertices of $G$.

Now we show that the labeling $c$ is a radio $k$-labeling of $G$, i.e., the inequality (\ref{eqn1}) holds for every $i,j\in V(G)$. Note that the length of the pattern is $tk+1$ and clearly $d_G(i,j)>k$ for every $i,j$ with $\vert i-j\vert \geq tk+1$. Therefore it suffices to prove that there is no conflict in labeling $c$ between vertices in two consecutive copies of the pattern.
If $\vert c_i-c_j\vert>k$ then the inequality (\ref{eqn1}) trivially holds. Now we consider the following possibilities.
\vspace{3mm}

\bmi
\item[{\sl Case 1:}] $\vert c_i-c_j\vert = 0$. From the definition of the pattern given by Table \ref{toddkodd} it follows that $\vert i-j\vert =tk+1$. By Lemma \ref{ldist}, $d_G(i,j)> k$ and we are done.
\item[{\sl Case 2:}] $\vert c_i-c_j \vert =k$. Clearly $d_G(i,j)>0$ and then $\vert c_i-c_j \vert + d_G(i,j)>k$.
\item[{\sl Case 3:}] $0<\vert c_i-c_j \vert<k$. From Table \ref{toddkodd} we have $c_j=c_i \pm \frac{k\pm 1}2$. We consider the following possibilities.
\begin{itemize}
\item[a)] $\vert c_i-c_j \vert = \frac{k-1}2$. From Table \ref{toddkodd} we obtain $\vert i-j\vert = \frac t2 k + \frac 12 = \frac{k-1}2 t + \frac t2 + \frac 12$. From Lemma~\ref{ldist} for $q=\frac {k-1}2$ and $r=\frac t2 + \frac 12$, it follows that $d_G(i,j)=\min \left\{ \frac {k-1}2 + \frac t2+ \frac 12; \, \frac{k-1}2 + 1 + t - (\frac t2+\frac 12 ) \right\}$. Then $d_G(i,j)>\frac {k+1}2$ for $t>1$. Thus we have $\vert c_i-c_j \vert + d_G(i,j)>\frac {k-1}2+\frac {k+1}2 =k$.
\item[b)] $\vert c_i-c_j \vert = \frac{k+1}2$. From Table \ref{toddkodd} we have $\vert i-j \vert = \frac t2 k + \frac 12 \pm 1$.
\begin{itemize}
\item[-] $\vert i-j\vert =\frac t2 k +\frac 12 -1 = \frac{k-1}2 t +\frac t2 - \frac 12$. By Lemma \ref{ldist} for $q=\frac {k-1}2$ and $r=\frac t2 - \frac 12$, it follows that $d_G(i,j)=\min \left\{ \frac {k-1}2 + \frac t2- \frac 12; \, \frac{k-1}2 + 1 + t - (\frac t2-\frac 12)  \right\}$. Hence we have $d_G(i,j)>\frac {k-1}2$ for $t>1$.
\item[-] $\vert i-j \vert = \frac t2 k + \frac 12 +1 = \frac{k-1}2 t +\frac t2 + \frac 32$. If $t=3$ then $q=\frac{k-1}2 +1$ and $r=0$ and, by Lemma \ref{ldist}, $d_G(i,j)= \frac{k-1}2 +1$. If $t>3$ then,
by Lemma \ref{ldist} for $q=\frac {k-1}2$ and $r=\frac t2 + \frac 32$, it follows that $d_G(i,j)=\min \left\{ \frac {k-1}2 + \frac t2+ \frac 32; \, \frac{k-1}2 + 1 + t - (\frac t2+\frac 32)  \right\}$. Thus $d_G(i,j)>\frac {k-1}2$ for $t>1$.
\end{itemize}
Now we have $\vert c_i-c_j \vert + d_G(i,j)>\frac {k+1}2+\frac {k-1}2 =k$ and we are done.
\end{itemize}
\emi
We have shown that the defined labeling is a radio $k$-labeling of $G$. The maximum used label is $\frac{k-1}2 + \left( \frac t2 k - \frac 12 \right)k=\frac t2 k^{2}-\frac 12$.
\end{proof}

\begin{Theorem} \label{thm D(1,t) t even}Let $t\geq 4$ be an even integer and $k\geq 1$ be and odd integer, let $G$ be a distance graph $D(1,t)$. Then
$\text{rl}_{k}(G)\leq \frac t2 k^{2}$.
\end{Theorem}

\begin{proof}
First we define a periodical pattern of labels. For vertices $1,2,\dots, tk+1$ of the distance graph $G$ we set labels using the following table.

\begin{table}[ht]
\centering
\small
$$\begin{array}{|c|c|c|c|c|c|c|c|c|c|c|c|c|} 
\hline
\mbox{vertex} & 1 & 2 & 3 & 4 & \dots & \frac t2 k+ 1 & \frac t2 k+2 & \frac t2 k+3 & \frac t2 k+4 &\dots & tk+1 \\
\hline
\mbox{label} & 0 & k & 2k & 3k & \dots & \left( \frac t2 k \right) k & \frac {k-1}2 & \frac {k-1}2 +k & \frac {k-1}2 + 2k & \dots & \frac {k-1}2 + \left( \frac t2 k -1\right) k \\
\hline
\end{array}$$
\normalsize
\caption{\label{tevenkodd}A periodical pattern for $G=D(1,t)$ and even $t$.}
\end{table}

Then we can define a labeling $c$ of all vertices of $G$ setting $c(a+b(tk+1))=c(a)$, $ a\in \{1,2,\dots, tk+1\}$ and $b\in \mathbb{Z}$, i.e., we repeat the defined periodical pattern for all vertices of $G=D(1,t)$.

Now we show that the labeling $c$ is a radio $k$-labeling of $G$, i.e., the inequality (\ref{eqn1}) holds for every $i,j\in V(G)$. Note that the length of the pattern is $tk+1$ and clearly $d_G(i,j)>k$ for every $i,j\in \mathbb{Z}$ with $\vert i-j\vert \geq tk+1$. Therefore it suffices to prove that there is no conflict in labeling $c$ between vertices in two consecutive copies of the pattern.
If $\vert c_i-c_j\vert>k$ then the inequality (\ref{eqn1}) trivially holds. Now we consider the following possibilities.

\vspace{3mm}

\bmi
\item[{\sl Case 1}:] $\vert c_i-c_j\vert = 0$. From the definition of the pattern given by Table \ref{tevenkodd} it follows that $\vert i-j\vert =tk+1$. Then $d_G(i,j)> k$ by Lemma \ref{ldist} and the inequality (\ref{eqn1}) holds. 
\item[{\sl Case 2}:] $\vert c_i-c_j \vert =k$. Clearly $d_G(i,j)>0$ and then $\vert c_i-c_j \vert + d_G(i,j)>k$.
\item[{\sl Case 3}:] $0<\vert c_i-c_j \vert<k$. From Table \ref{tevenkodd} we have $c_j=c_i \pm \frac{k\pm 1}2$, i.e., $\vert c_i-c_j\vert =  \frac{k\pm 1}2$. From Table \ref{tevenkodd} we also obtain $\vert i-j\vert =\frac t2 k =  \frac{k-1}2 t+ \frac t2$ or $\vert i-j\vert =\frac t2 k + 1=\frac{k-1}2 t +\frac t2+1$. Suppose that $\vert i-j\vert =\frac{k-1}2 t+\frac t2$. Then, by Lemma \ref{ldist}, $d_G(i,j)= \min \left\{ \frac {k-1}2 +\frac t2; \frac{k-1}2+1+ t-\frac t2 \right\}$. For $t\geq 4$ we get $d_G(i,j)>\frac{k+1}2$. Now suppose that $\vert i-j\vert =\frac{k-1}2 t+\frac t2+1$. Then, by Lemma \ref{ldist}, $d_G(i,j)= \min \left\{ \frac {k-1}2 +\frac t2+1; \frac{k-1}2+1+t-\frac t2-1 \right\}$. For $t\geq 4$ we get $d_G(i,j)>\frac{k+1}2$. Thus we have shown that, for every even $t\geq4$, $\vert c_i-c_j\vert + d_G(i,j)>\frac {k-1}2+\frac {k+1}2=k$.
\emi
We have shown that the defined labeling is a radio $k$-labeling of $G$. Clearly the maximum used label is $\left( \frac t2 k  \right)k=\frac t2 k^{2}$.
\end{proof}

\subsection{$D(t-1,t)$}

Now we consider the distance graph $G=D(t-1,t)$. For even $k$ we did not find any improvement of the upper bound for $\text{rl}_k(G)$ given in Theorem \ref{thmkeven}. But for odd $k$ we prove the following statement which decreases the upper bound for $\text{rl}_k(G)$ given by Theorem \ref{thmkodd}.

\begin{Theorem} \label{thm D(t-1,t)}
Let $t> 2$ be an integer and $k\geq 3$ be and odd integer, let $G$ be a distance graph $D(t-1,t)$. Then
$\text{rl}_{k}(G)\leq \frac t2 k^{2}+k-\frac{t+2}{2}$. 
\end{Theorem}

\begin{proof}
First we define a periodical pattern of labels. For vertices $1,2,\dots, tk+t+3$ of the distance graph $G$ we set labels using Table \ref{t-1,t}, where $l=k-1$.

\begin{table}[ht]
\centering
\scriptsize
$$\begin{array}{|c|c|c|c|c|c|c|c|c|c|c|c|c|} 
\hline
\mbox{vertex} & 1 & 2 & 3 & 4 & \dots &  \frac t2 k+\frac{t+2}2+ 1 & \frac t2 k+\frac{t+2}2+2 & \frac t2 k+\frac{t+2}2+3 & \frac t2 k+\frac{t+2}2+4 &\dots & tk+t+3 \\
\hline
\mbox{label} & 0 & l & 2l & 3l & \dots & \left( \frac t2 k+\frac{t+2}2 \right) l & \frac {l}2 & \frac {l}2 +l & \frac {l}2 + 2l & \dots & \frac {l}2 + \left( \frac t2 k +\frac t2\right) l \\
\hline
\end{array}$$
\normalsize
\caption{\label{t-1,t}A periodical pattern for $G=D(t-1,t)$ and odd $k$.}
\end{table}

Then we can define a labeling $c$ of all vertices of $G$ setting $c(a+b(tk+t+3))=c(a)$, $ a\in \{1,2,\dots, tk+t+3\}$ and $b\in \mathbb{Z}$, i.e., we repeat the defined periodical pattern for all vertices of $G=D(t-1,t)$.

Now we show that the labeling $c$ is a radio $k$-labeling of $G$, i.e., the inequality (\ref{eqn1}) holds for every $i,j\in V(G)$. Note that the length of the pattern is $tk+t+3$ and clearly $d_G(i,j)>k$ for every $i,j$ with $\vert i-j\vert \geq tk+t+3$. Therefore it suffices to prove that there is no conflict in labeling $c$ between vertices in two consecutive copies of the pattern. If $\vert c_i-c_j\vert >k$ then the inequality (\ref{eqn1}) trivially holds. The following possibilities can occur.

\bmi
\item[{\sl Case 1:}] $\vert c_i-c_j\vert =0$. From the definition of the pattern given by Table \ref{t-1,t}, it follows that $\vert i-j\vert>tk$, implying that the inequality (\ref{eqn1}) holds.
\item[{\sl Case 2:}] $\vert c_i-c_j\vert =k-1$. It follows that $\vert i-j\vert=1$. For $t>2$ we have $d_G(i,j)>1$, implying that the inequality (\ref{eqn1}) holds.
\item[{\sl Case 3:}] $\vert c_i-c_j\vert = \frac{k-1}2$. By the definition of the pattern given by Table \ref{t-1,t}, $\vert i-j\vert =\frac t2 k+\frac t2+1$ or $\vert i-j\vert = \frac t2 k+\frac t2+2$. Since $\vert i-j\vert > \frac{k+1}2 t$ and $k$ is odd, we get $d_G(i,j)>\frac{k+1}2$, implying that $\vert c_i-c_j\vert + d_G(i,j)>\frac{k-1}2+\frac{k+1}2=k$.
\emi
The maximum used label in $c$ is $\left(\frac t2 k +\frac {t+2}{2} \right)(k-1)=\frac t2 k^2+k-\frac{t+2}2$. 

\end{proof}

\section{Values and bounds for $\mbox{rl}_k(D(1,2,\dots, t)$, $\mbox{rl}_k(D(1,t))$ and \\ $\mbox{rl}_k(D(t-1,t))$ for small $k$, $t$. }

Lower and upper bounds on radio $k$-labeling number of the distance graphs $D(1,2,\dots, t)$, $D(1,t)$ and $D(t-1,t)$ can be obtained from theorems and propositions given in Sections 2 and 3. For small values $t,k \in \{2,\dots,9 \}$, we improve these bounds using a computer. For finding lower bounds we used brute force search program. The program takes vertices $X=\{ 1,2,\dots,i \}$ of the distance graph $G$ and it tries to construct a radio $k$-labeling $c$ of $X$ using labels $0,\dots, l$. First it assigns label $0$ to vertex $1$ (there must be a vertex with label 0, otherwise we can decrease all labels to get smaller bound) and tries to extend $c$ to $X$. If the extension is not possible we conclude that $\mbox{rl}_k(G)>l$. 

For finding upper bounds, we found and verified (again using computer) patterns, which can be periodically repeated for a whole distance graph $G$.

The lower and upper bounds shown in Tables \ref{tab1}, \ref{tab2} and \ref{tab3} are presented at the web pages \texttt{http://home.zcu.cz/\~{}holubpre/radio\underline{ }labeling/} and were computed on the Metacentrum computing facilities. 

We end this section by presenting some lower and upper bounds for $\text{rl}_k(D(1,2,\dots,t))$, $\text{rl}_k(D(1,t))$ and $\text{rl}_k(D(t-1,t))$  for small positive integers $k$ and $t$ in the following tables. The emphasized numbers are exact values, all the pairs of values are lower and upper bounds.

\begin{table}[ht]
\centering
$$ \begin{array}{|c|c|c|c|c|c|c|c|c|}
\hline
t \diagdown k & 2 & 3 & 4 & 5 & 6 & 7 & 8 & 9   \\
\hline
2 & {\bf 6} & {\bf 12} & {\bf 20} & {\bf 30} & {\bf 42} &  {\bf 56} & 65-72 & 82-90   \\
\hline
3 & {\bf 8} & {\bf 17} & {\bf 28} & {\bf 43} & 55-60  &   74-81 & 97-104 & 122-135  \\
\hline
4 & {\bf 10} & {\bf 22} & {\bf 36} & 51-56 & 73-78   & 99-112 & 129-136 & 163-180 \\
\hline
5 & {\bf 12} &{\bf 27} & {\bf 43} & 63-69 & 91-96   & 123-131 & 161-168 & 203-217 \\
\hline
6 & {\bf 14} & {\bf 32} & 49-52 & 76-82 &  109-114  &   148-163 &  193-200 & 244-259  \\
\hline
7 & {\bf 16} & 32-37 &  57-60 &  88-95 &  127-132  &  172-189 &  225-232 &  284-301 \\
\hline
8 & {\bf 18} & 37-42 &  65-68 &  101-108 & 145-150 &   197-215 & 257-264 &  325-343 \\
\hline
9 & {\bf 20} & 41-47 &  73-76 & 113-121 &  163-168 &   221-241 &  289-296 &  365-385 \\
\hline
\end{array}$$
\caption{\label{tab1} Values and bounds for $\text{rl}_k(D(1,2,\dots,t))$ for small $k,t$.}
\end{table}

\begin{table}[ht]
\centering
$$\begin{array}
{|c|c|c|c|c|c|c|c|c|}
\hline
t \diagdown k & 2 & 3 & 4 & 5 & 6 & 7 & 8 & 9  \\
\hline
2 & {\bf 6} & {\bf 12} & {\bf 20} & {\bf 30} & {\bf 42} & {\bf 56} & 65- 72 & 82- 90  \\
\hline
3 & {\bf 6} & {\bf 11} & {\bf 24} & {\bf 33} & 51-  52 & 61 -73 & 81- 100 & 105- 121   \\
\hline
4 & {\bf 6} & {\bf 15} & {\bf 26} & {\bf 43} & 54- 64 & 69- 94 & 95- 116 & 124- 152 \\
\hline
5 &{\bf 6} &{\bf 13} & {\bf 26}& {\bf 41} & 49- 66 & 73- 91 & 103- 140 & 137- 165 \\
\hline
6 & {\bf 7} & {\bf 14} & {\bf 28} &  41- 48 &  46- 72 &  73- 102 & 105 -147 &   144- 196 \\
\hline
7 & {\bf 7} & {\bf 12} & {\bf 26} & {\bf 37} &  42- 78 &  69- 111 & 104-146 &  145-201 \\
\hline
8 & {\bf 7} & {\bf 13} & {\bf 26} & 36- 48 & 46- 75 & 62- 116 & 98- 159 & 141- 212 \\
\hline
9 & {\bf 6} & {\bf 11} &  25- 28 &  32- 41 &  40- 74 &  54- 99 & 89 - 156 & 134 - 207 \\
\hline
\end{array}$$
\caption{\label{tab2}Values and bounds for $\text{rl}_k(D(1,t))$ for small $k,t$.}
\end{table}

\begin{table}[ht]
\centering
$$\begin{array}
{|c|c|c|c|c|c|c|c|c|}
\hline
t \diagdown k & 2 & 3 & 4 & 5 & 6 & 7 & 8 & 9   \\
\hline
2 & {\bf 6} & {\bf 12} & {\bf 20} & {\bf 30} & {\bf 42} &  {\bf 56} & 65-72 & 82-90   \\
\hline
3 & {\bf 6} & {\bf 14} & {\bf 28} & {\bf 40} & 51-60  &   60-78 & 66-104 & 71-128  \\
\hline
4 & {\bf 7} & {\bf 14} & {\bf 27} & 40-48 & 50-70   & 57-99 & 65-131 & 71-166 \\
\hline
5 & {\bf 6} &{\bf 13} & {\bf 26} & 37-50 & 47-78 & 55-116 & 61-150 & 70-191 \\
\hline
6 & {\bf 6} & {\bf 14} & 24-30 & 32-54 & 40-69 &  50-108 & 60-153 & 65-208  \\
\hline
7 & {\bf 7} & {\bf 13} & 21-28 & 28-48&  36-74 &  45-109 &  55-148 &  60-223 \\
\hline
8 & {\bf 7} & {\bf 13} &  19-26 &  24-48 & 32-73 & 40-120 & 49-148 &  56-211 \\
\hline
9 & {\bf 6} & {\bf 12} &  17-27 &  23-53 & 30-70 &   36-109 & 44-164 &  51-233 \\
\hline
\end{array}$$
\caption{\label{tab3}Values and bounds for $\text{rl}_k(D(t-1,t))$ for small $k,t$.}
\end{table}

\section*{Acknowledgement}

The first three authors were supported by the Center of Excellence -- Institute for Theoretical Computer Science, Prague (project P202/12/G061 of Grant Agency of the Czech Republic). The last author was supported by Burgundy University under project BQR 036 (2011).

This work was supported by the European Regional Development Fund (ERDF), project "NTIS - New Technologies for the Information Society", European Centre of Excellence, CZ.1.05/1.1.00/02.0090.

The access to the MetaCentrum computing facilities provided under the programme ``Projects of Large Infrastructure for Research, Development and Innovations" LM2010005 funded by the Ministry of Education, Youth, and Sports of the Czech Republic is highly appreciated.

\clearpage

\end{document}